\documentclass{amsart}
\catcode`\é=\active
\defé{\'e}

\catcode`\è=\active
\defè{\`e}

\catcode`\ê=\active
\defê{\^e}

\catcode`\û=\active
\defû{\^u}

\catcode`\â=\active
\defâ{\^a}

\catcode`\à=\active
\defà{\`a}

\catcode`\ù=\active
\defù{\`u}

\catcode`\ç=\active
\defç{\c c}

\catcode`\î=\active
\defî{\^{\i}}

\catcode`\ï=\active
\defï{\"{\i}}

\catcode`\ô=\active
\defô{\^o}

\catcode`\é=\active
\defé{\'e}

\def\nb{{\Bbb N}}

\def\rb{{\Bbb R}}

\usepackage{amssymb}
\usepackage{url}
\usepackage{amscd}
\usepackage{graphicx}
\newtheorem{theorem}{Theorem} % 1st argument is your name for it
\newtheorem{proposition}{Proposition}
\begin{document}
\title[On the zeros of Dirichlet $L$-functions]{On the zeros of Dirichlet $L$-functions} % This is the full title of the paper
% Use lowercase letters in title except for proper names
% Avoid equations in title if possible
% Do not use the \thanks{} command; use \extraline{} instead (see below).

\author{Sami Omar, Raouf Ouni, Kamel Mazhouda}
\address{Faculty of Science of Tunis, Department of Mathematics, 2092 Tunis, Tunisia}
\email{sami.omar@fst.rnu.tn}
\address{Faculty of Science of Tunis, Department of Mathematics, 2092 Tunis, Tunisia}
\email{raouf.ouni@fst.rnu.tn}
\address{Faculty of Science of Monastir, Department of Mathematics, 5000 Monastir, Tunisia}
\email{kamel.mazhouda@fsm.rnu.tn}

%\dedication{A dedication can be included here}

%Insert `2000 Mathematics Subject Classification' numbers here:
\thanks{2000 Mathematics Subject Classification: 11M06,11M26, 11M36.\\Keywords and Phrases: Dirichlet $L$-functions, Li's criterion, Riemann hypothesis.}

\maketitle

\begin{abstract}
In this paper, we compute and verify the positivity of  the Li coefficients for the Dirichlet $L$-functions  using an arithmetic formula  established in
Omar and Mazhouda, J. Number Theory 125 (2007) no.1, 50-58;  J. Number Theory 130 (2010) no.4, 1109-1114. Furthermore, we formulate a criterion for the partial
Riemann hypothesis and we provide some numerical evidence for it using new formulas for the Li coefficients.

\end{abstract}

%\tableofcontents
\begin{section}{Introduction}\label{sec.1}
The Li criterion  for the Riemann hypothesis (see. \cite{6}) is a necessary and sufficient condition that the sequence $$\lambda_{n}=\sum_{\rho}\left[1-\left(1-\frac{1}{\rho}\right)^{n}\right]$$ is non-negative for all $n\in{\mathbb {N}}$ and where $\rho$ runs over the non-trivial zeros of $\zeta(s)$. This criterion holds also for the Dirichlet $L$-functions and for a large class of Dirichlet series, the so called the  Selberg class as given in \cite{10}. More recently, Omar and Bouanani \cite{9} extended the Li criterion for function fields and established an explicit and asymptotic formula for the Li coefficients.\\

Numerical computation of the first 100 of the Li coefficients $\lambda_{n}$ which appear in this criterion was made by Maslanka \cite{8}  and later by Coffey in \cite{2},  who computed and verified  the positivity of about 3300 of the Li coefficient $\lambda_{n}$. The main empirical observation made by Maslanka is that these coefficients can be separated in two parts, where one of them grows smoothly while the other one is very
small and oscillatory. This apparent smallness is quite unexpected. If it persisted until infinity then
the Riemann hypothesis would be true. As we said above, this criterion was  extended to a large class of Dirichlet series \cite{10} and no calculation or verification of the positivity to date in the literature made for other $L$-functions.\\

In this paper, we compute and verify the positivity of  the Li coefficients for the Dirichlet $L$-functions  using an arithmetic formula  established in \cite{10,11}. Furthermore, we formulate a criterion for the partial
Riemann hypothesis. Additional results are presented, including  new formulas for the Li coefficients. Basing on the numerical computations made below, we conjecture that these coefficients are increasing in $n$.  Should this conjecture hold, the validity of the Riemann hypothesis would follow.\\

Next, we review the Li criterion for the case of the Dirichlet $L$-functions. Let $\chi$ be a primitive Dirichlet character of conductor $q$. The Dirichlet $L$-function
attached to this character is defined by
$$L(s, \chi)=\sum_{n=1}^{\infty}\frac{\chi(n)}{n^{s}}, \ \ \ \ (Re(s) > 1).$$
For the trivial character $\chi = 1$, $L(s, \chi)$ is the Riemann zeta function. It is well
known \cite{3} that if $\chi\neq1$ then $L(s,\chi)$ can be extended to an entire function in
the whole complex plane and satisfies the functional equation
$$\xi(s,\chi)= \omega_{\chi}\xi(1-s,\overline{\chi}),$$
where $$\xi(s,\chi)=\left(\frac{q}{\pi}\right)^{(s+a)/2}\Gamma\left(\frac{s+a}{2}\right)L(s,\chi),$$
$$a=\left\{\begin{array}{crll}0&\hbox{if}&\chi(-1)= 1\\1&\hbox{if}&\chi(-1) =-1,\end{array}\right.\ \ \ \ \hbox{and}\ \ \ \  \omega_{\chi}=\frac{\tau(\chi)}{\sqrt{q}i^{a}},$$
where $\tau(\chi)$ is the Gauss sum $$\tau(\chi)=\sum_{m=1}^{q}\chi(m)e^{2\pi im/q}.$$ The function $\xi(s,\chi)$ is an entire function of order one. The function $\xi(s,\chi)$ has  a product representation
\begin{equation}\label{eq.1}\xi(s,\chi) =\xi(0,\chi)\prod_{\rho}\left(1-\frac{s}{\rho}\right),\end{equation}
where the product is over all the zeros of $\xi(s,\chi)$ in the order given by $|Im(\rho)|<T$ for $T\rightarrow\infty$.
If $N_{\chi} (T)$ counts the number of zeros of $L(s,\chi)$ in the rectangle   $0\leq Re(s)\leq1$, $0<Im(s)\leq T$ (according to multiplicities) one can show by standard contour integration the formula
$$N_{\chi}(T)=\frac{1}{2\pi}T\log T+c_{1}T+O\left(\log T\right),$$
where
$$c_{1}=\frac{1}{2\pi}\left(\log q-\left(\log(2\pi)+1\right)\right).$$
 We put
$$\lambda_{\chi}(n)=\sum_{\rho}\left[1-\left(1-\frac{1}{\rho}\right)^{n}\right],$$
where the sum over $\rho$ is $\sum_{\rho}=\lim_{T\mapsto\infty}\sum_{|Im\rho|\leq T}$.\\
{\bf Li's criterion} says that  $\lambda_{\chi}(n)> 0$ for all $n = 1, 2, .$ . . if and only if all of the zeros of $\xi(s,\chi)$ are located on the critical line $Re(s) = 1/2$.\\

The paper is organized as follows. In Section \ref{sec.2},  we recall the arithmetic formula for the Li coefficients for the Dirichlet $L$-functions and  we give an estimate for the error term of $\lambda_{\chi}(n)$. In Section \ref{sec.3}, we show that $\lambda_{\chi}(n)\geq0$ if every non-trivial zero of $L(s,\chi)$ with $|Im(\rho)|<\sqrt{n}$ satisfies  $Re(\rho)=1/2$ (that is partial Riemann hypothesis) and we give  an estimate for the difference $|\lambda_{\chi}(n)-\lambda_{\chi}(n,T)|$, where $\lambda_{\chi}(n,T)$ are the partial Li coefficients. In Section \ref{sec.4}, we prove new formulas (integral and summation formula) for the Li coefficients $\lambda_{\chi}(n)$. Finally, in Section \ref{sec.5}, we report numerical computations of the Li coefficients using different formulas established in the previous sections unconditionally or under the Riemann hypothesis.
\end{section}
\section{Li's coefficients}\label{sec.2}

Applying \cite[Theorem 2.2]{10} for the case of the Dirichlet $L$-functions, we get the following arithmetic formula.
\begin{theorem}\label{th.1}Let $\chi$ be a primitive Dirichlet character of conductor $q>1$. We have
\begin{eqnarray}\label{eq.2}
\lambda_{\chi}(n)&=&-\sum_{j=1}^{n}(_{j}^{n})\frac{(-1)^{j-1}}{(j-1)!}\sum_{k=1}^{+\infty}\frac{\Lambda(k)}{k}\chi(k)(\log k)^{j-1}\nonumber\\
&&\ \ +\ \ \frac{n}{2}\left(\log\frac{q}{\pi}-\gamma\right)+\tau_{\chi}(n)
\end{eqnarray}
where
$$\tau_{\chi}(n)=\left\{\begin{array}{crll}\sum_{j=2}^{n}(_{j}^{n})(-1)^{j}\left(1-\frac{1}{2^{j}}\right)\zeta(j)-\frac{n}{2}\sum_{l=1}^{+\infty}\frac{1}{l(2l-1)}&\hbox{if}&\chi(-1)=1,\\ \sum_{j=2}^{n}(_{j}^{n})(-1)^{j}2^{-j}\zeta(j)&\hbox{if}&\chi(-1)=-1.\end{array}\right.$$
\end{theorem}
Theorem \ref{th.1} is also proved by Coffey in \cite{2} and Li in \cite{7}.
The arithmetic formula above can be written as
\begin{eqnarray}\label{eq.3}
\lambda_{\chi}(n)&=&\left[\log\frac{q}{\pi}+\psi\left(\frac{a+1}{2}\right)\right]\frac{n}{2}+\sum_{j=2}^{n}(_{j}^{n})\frac{1}{(j-1)!}2^{-j}\psi^{(j-1)}\left(\frac{a+1}{2}\right)\nonumber\\
&&-\ \sum_{j=1}^{n}(_{j}^{n})\frac{(-1)^{j-1}}{(j-1)!}\sum_{k=1}^{+\infty}\frac{\Lambda(k)\chi(k)}{k}(\log k)^{j-1},\nonumber
\end{eqnarray}
 where $a=0$ if $\chi(-1)=1$ and 1 if $\chi(-1)=-1$, with  $\psi(\frac{1}{2})=-\gamma-2\log2$, $\psi(1)=-\gamma$,   \ $\gamma$ is the Euler constant, $\psi^{(j-1)}(1)=(-1)^{j}(j-1)!\zeta(j)$\ \  and $\psi^{(j-1)}(\frac{1}{2})=(-1)^{j+1}j!(2^{j+1}-1)\zeta(j+1).$ Here, $\psi=\frac{\Gamma'}{\Gamma}$ denotes the digamma function.\\

An asymptotic formula for the number $\lambda_{\chi}(n)$ was proved in \cite[Theorem 3.1]{12} using the arithmetic formula. Furthermore, it  is equivalent to the Riemann hypothesis.
\begin{theorem}\label{th.2} We have
$$RH\Leftrightarrow\lambda_{\chi}(n)=\frac{1}{2}n\log n+c_{\chi}n+O(\sqrt{n}\log n),$$
where $$c_{\chi}=\frac{1}{2}(\gamma-1)+\frac{1}{2}\log(q/\pi),$$
and $\gamma$ is the Euler constant.
\end{theorem}
Here, we estimate the error term for $\lambda_{\chi}(n)$ using the arithmetic formula (\ref{eq.2}).
Writing (\ref{eq.2}) in the form
$$\lambda_{\chi}(n)=\tilde{\lambda}_{\chi}(n,M)+E_{M},$$
where
\begin{equation}\label{eq.4}\tilde{\lambda}_{\chi}(n,M)=\frac{n}{2}\left(\log\frac{q}{\pi}-\gamma\right)+\tau_{\chi}(n)-\sum_{j=1}^{n}(_{j}^{n})\frac{(-1)^{j-1}}{(j-1)!}\sum_{k\leq M}\frac{\Lambda(k)}{k}\chi(k)(\log k)^{j-1}\end{equation}
and $$E_{M}=-\sum_{j=1}^{n}(_{j}^{n})\frac{(-1)^{j-1}}{(j-1)!}\sum_{k>M}\frac{\Lambda(k)}{k}\chi(k)(\log k)^{j-1}=-\sum_{k>M}\frac{\Lambda(k)}{k}\chi(k)L_{n-1}^{1}(\log k),$$
where $L_{n-1}^{1}$ is an associated Laguerre polynomial of degree $n-1$.\\

Next, for our computation in the section \ref{sec.5}, we need to find $M$ such that $|E_{M}|\leq 10^{-\nu}$. An  estimate for the Laguerre polynomials is due to Koepf and Schmersau \cite[Theorem 2]{5}. Actually,  they have shown that
 $$|L_{n}^{\alpha}(x)|<e^{x/2}\left[(n+\alpha)/x\right]^{\alpha/2}$$
 for $x\in{[0,4(n+\alpha)]}$ when $n+\alpha>0$ and $\alpha$ is an integer. Then, we obtain
   \begin{eqnarray}\label{eq.33}|E_{M}|&\leq&\left|\sqrt{\frac{n}{\log M}}\sum_{m>M}^{+\infty}\frac{\Lambda(m)}{\sqrt{m}}\chi(m)\right|\nonumber\\
 &\leq&\sqrt{\frac{n}{\log M}}\left\{\left|\sum_{p>M}^{+\infty}\frac{\log p}{\sqrt{p}}\chi(p)\right|+\left|\sum_{p^{2}>M}^{+\infty}\frac{\log p}{p}\chi(p^{2})\right|\right\}\nonumber\\
 &&+\ \ \sqrt{\frac{n}{\log M}}\left|\sum_{p^{j}>M,j>2}^{+\infty}\frac{\log p}{p^{j/2}}\chi(p^{j})\right|.
 \end{eqnarray}
 We have  $|\chi(p^{j})|<1$. Therefore,
 $$\left|\sum_{p>M}^{+\infty}\frac{\log p}{\sqrt{p}}\chi(p)\right|+\left|\sum_{p^{2}>M}^{+\infty}\frac{\log p}{p}\chi(p^{2})\right|\leq\left|\sum_{p>M}^{+\infty}\frac{\log p}{\sqrt{p}}\chi(p)\right|+\left|\sum_{p>\sqrt{M}}^{+\infty}\frac{\log p}{p}\chi(p)\right|\leq\frac{2}{\sqrt{M}}$$
 and
$$
\sum_{p^{j}>M,j>2}^{+\infty}\frac{\log p}{p^{j/2}}\leq\left\{\begin{array}{crll} &\frac{\log M}{\sqrt{M}}&    \hbox{if}\ \ M+1 \hbox{ is prime,}\\ &\frac{1}{\sqrt{M}}&   \hbox{otherwise.}\end{array}\right.
$$
Then
\begin{equation}\left\{\begin{array}{crll}|E_{M}|&\leq&\sqrt{\frac{n}{\log M}}\ \frac{\log M+2}{\sqrt{M}}\leq \sqrt{n}\ \frac{\log M+2}{\sqrt{M}}&  \hbox{if}\ M+1 \hbox{ is prime,}\\ |E_{M}|&\leq&\sqrt{\frac{n}{\log M}}\frac{3}{\sqrt{M}}\leq3\frac{\sqrt{n}}{\sqrt{M}}&\ \hbox{otherwise.}\end{array}\right.
\end{equation}
Let $M$ be such that $|E_{M}|\leq 10^{-\nu}$. Then,
%Then $$\left\{\begin{array}{crll}\sqrt{n}\ \frac{\log M+2}{\sqrt{M}}\leq 10^{-\nu} & \hbox{si}\  M+1 \ \hbox{ est  premier,}\\ 3\frac{\sqrt{n}}{\sqrt{M}}\leq 10^{-\nu}&\hbox{sinon.}\end{array}\right.$$
 using the theory of the Lambert $W$ function, we choose  $M$
\begin{equation}
\left\{\begin{array}{crll}M&=&\frac{n}{4}\ \left[W_{-1}\left(-\frac{10^{-\nu}}{\sqrt{n}}\right)\right]^{2}+4n\ 10^{2\nu}& \hbox{if}\ M+1\ \hbox{ is prime,}\\ M&=&9n\ 10^{2\nu}&\hbox{otherwise,}\end{array}\right.
\end{equation}
where $W_{-1}$ denotes the branch satisfying $W(x)\leq-1$ and $W(x)$ is the Lambert $W$ function,
which is defined to be the multivalued inverse of the function  $w\longmapsto we^{w}$.
 \begin{section}{Partial Li criterion}\label{sec.3}
 In the following proposition, we  propose a partial Li criterion which relates the partial Riemann hypothesis to the positivity of the Li coefficients up to a certain order.
\begin{proposition}\label{prop.1} For sufficiently large $T\geq T_{1}\geq1$, if every non-trivial zero $\rho$ of $L(s,\chi)$  with $|Im(\rho)|<T$ satisfies $Re(\rho)=1/2$, then $\lambda_{\chi}(n)\geq0$ for all $n\leq T^{2}$.
\end{proposition}
\begin{proof}  We have
$$\lambda_{\chi}(n)=\sum_{\rho}\left[1-\left(1-\frac{1}{\rho}\right)^{n}\right].$$
Then
\begin{eqnarray}\label{eq.8}
\lambda_{\chi}(n)&=&\sum_{\rho}\left(1-Re\left[\left(1-\frac{1}{\rho}\right)^{n}\right]\right)\nonumber\\
&=&\sum_{\rho; |Im(\rho)|<T}\left(1-Re\left[\left(1-\frac{1}{\rho}\right)^{n}\right]\right)+\sum_{\rho; T<|Im(\rho)|}\left(1-Re\left[\left(1-\frac{1}{\rho}\right)^{n}\right]\right)\nonumber\\
\end{eqnarray}
Let  $\rho=\beta+i\gamma$, then we obtain
$$1-\left(1-\frac{1}{\rho}\right)^{n}=1-\left(\frac{1+\frac{\beta-1}{i\gamma}}{1+\frac{\beta}{i\gamma}}\right)^{n},\ \frac{1+\frac{\beta-1}{i\gamma}}{1+\frac{\beta}{i\gamma}}=\left(1-\frac{\beta}{\gamma^{2}}\right)+\frac{i}{\gamma}+O\left(\frac{1}{\gamma^{3}}\right)$$
and using binomial identity, we get
%\footnote{The estimation of the absolute value of the difference between the fraction
%$$Re\left[\left(\frac{1+\frac{\beta-1}{i\gamma}}{1+\frac{\beta}{i\gamma}}\right)^{n}\right] \ \hbox{and the value}\  \left(1-\frac{\beta}{\gamma^{2}}\right)^{n}$$
 %under the assumptions $\gamma\geq n$ and $\gamma\gg 1$ has the form $\ll c\left(\frac{n^{2}}{\gamma^{3}}\right)$ with $c=  0,53$. In particular, it depends on the parameter $n$.}
$$Re\left[\left(\frac{1+\frac{\beta-1}{i\gamma}}{1+\frac{\beta}{i\gamma}}\right)^{n}\right]=\left(1-\frac{\beta}{\gamma^{2}}\right)^{n}+O\left(\frac{1}{\gamma^{3}}\right),$$
where the $O$-symbol depends on $n$. Therefore
\begin{eqnarray}\label{eq.9}\sum_{ |\gamma|>T}\left(1-Re\left[\left(1-\frac{1}{\rho}\right)^{n}\right]\right)&=&\sum_{ |\gamma|>T}\left[1-\left(1-\frac{\beta}{\gamma^{2}}\right)^{n}+O\left(\frac{1}{\gamma^{3}}\right)\right]\nonumber\\
&=&\sum_{ |\gamma|>T}\frac{n\beta}{\gamma^{2}}+O\left(\frac{1}{\gamma^{3}}\right).
\end{eqnarray}
Then, the second sum in \eqref{eq.8} goes to 0 in absolute value as $T \to \infty$ by conditional convergence (see also Proposition \ref{prop.2} below).
%\begin{eqnarray}\label{eq.10}\sum_{ T<|\gamma|<T\log T}\left(1-Re\left[\left(1-\frac{1}{\rho}\right)^{n}\right]\right)&=&\sum_{ T<|\gamma|<T\log T}\left[1-\left(1-\frac{\beta}{\gamma^{2}}\right)^{n}+O\left(\frac{1}{\gamma^{3}}\right)\right]\nonumber\\
%&=&\sum_{ T<|\gamma|<T\log T}\frac{n\beta}{\gamma^{2}}+O\left(\frac{1}{\gamma^{3}}\right)\nonumber\\
%&=&O\left((\log T)^{2}\right).
%\end{eqnarray}
%Indeed, since $n\leq T^{2}$, we obtain
%\begin{eqnarray}
%\sum_{ T<|\gamma|<T\log T}\frac{n\beta}{\gamma^{2}}&\ll& \left[\sum_{j=T}^{T\log T}\frac{n}{j^{2}}\right]\nonumber\\
%&\ll & T^{2}\left[\sum_{j=\sqrt{T}}^{\sqrt{T\log T}}\frac{1}{j}\right]\nonumber\\
%&\ll& T^{2}\left[\log\log T\right]\nonumber\\
%&\ll&T^{2}\log\log T.\nonumber
%\end{eqnarray}
Finally, it suffices to prove that, under the Riemann hypothesis, there exists a positive constant $c_{0}$  such that for large $T$ we have
 %\footnote{Hier, we  need the existence and not  the precise values of $c_{0}$. The proof may be allows  to conclude that the corresponding assertion is true for some $T\geq T_{0}>0$. The problem of the evaluating the constant $T_{0}$ stays open.}
\begin{equation}\label{eq.11}
\sum_{|Im(\rho)|<T}\left(1-Re\left[\left(1-\frac{1}{\rho}\right)^{n}\right]\right)\geq c_{0} \frac{\log T}{T}.
\end{equation}
Thus, for $1\leq n\leq T^{2}$, we get
\begin{eqnarray}\label{eq.12}\sum_{ |\gamma|<T}\left(1-Re\left[\left(1-\frac{1}{\rho}\right)^{n}\right]\right)&\geq&\sum_{ |\gamma|<T}\frac{n^{2}}{2\gamma^{2}}\nonumber\\
&\geq&\frac{n^{2}}{2}\sum_{ |\gamma|<T}\frac{1}{\gamma^{2}}\nonumber\\
&\geq&\frac{n^{2}}{2T^{2}}\sum_{ |\gamma|<T}1\nonumber\\
&\geq&\frac{1}{2T^{2}}N_{\chi}(T).
\end{eqnarray}
Recall that
$$N_{\chi}(T)=\frac{1}{2\pi}T\log T+c_{1}T+O\left(\log T\right).$$
Then, equation \eqref{eq.11} is proved.
\end{proof}

\noindent{\bf Remark}
\begin{itemize}
	\item Recall that the $10^{13}$ first zeros of the Riemann zeta function lie on the line $Re(s)=1/2$ (see. \cite{4}). Then, from  Proposition \ref{prop.1}, we might expect that the first $10^{26}$ Li coefficients $\lambda_{\zeta}(n)$ are non-negative.

	\item In the section \ref{sec.5}, we will  use the first $10^{4}$ critical zeros of the
Dirichlet $L$-functions to compute the first Li coefficients $\lambda_{\chi}(n)$. Then, from Proposition \ref{prop.1} above, we also might expect   that the first $10^{8}$ Li coefficients are non-negative.
\end{itemize}
\noindent{\bf Conversely.} From the work of Brown \cite{1}, the first observation is that the first "non-trivial" inequality $\lambda_{\chi}(2)\geq0$ is sufficient to establish the non-existence of a Siegel zero for $\xi(s,\chi)$ (see \cite[Corollary 1]{1}). \\
\ \ Let $r>1$ be a real number. By the invariance of the zeros $\rho$ of $\xi(s,\chi)$ under the map $\rho\longmapsto1-\overline{\rho}$,
$$\forall \ \rho,\ \ \left|\frac{\rho}{\rho-1}\right|\leq r\ \Leftrightarrow\ \forall\ \rho,\ \rho\in{D_{r}},$$
where $D_{r}$ is the closed region bounded by the lines $\{z\in{\Bbb C}:\ Re(z)=0,1\}$ and the  arcs of tow circles. The second observation (see \cite[Theorem 3]{1}) is that, for large $N\in{\Bbb N}$, the equality $\lambda_{\chi}(1)\geq0,...,\lambda_{\chi}(n)\geq0$ imply the existence of a certain zero-free region for  $\xi(s,\chi)$, that is, there exist constants $N,\mu, \nu$ depending only on $q$ such that if $\lambda_{\chi}(1)\geq0,...,\lambda_{\chi}(n)\geq0$ hold, and $n\geq N$, then the zeros of $\xi(s,\chi)$ belong to $D_{r}$, where $r=\sqrt{1+T^{-2}}$ and $T=\left(\frac{n}{\mu\log^{2}(\nu n)}\right)^{1/3}$.\\

 Let us define the partial Li coefficients by
$$\lambda_{\chi}(n,T)=\sum_{\rho;\ |Im\rho|\leq T}1-\left(1-\frac{1}{\rho}\right)^{n}$$ with a parameter $T$. An estimate for  the error term $|\lambda_{\chi}(n)-\lambda_{\chi}(n,T)|$ is stated in the following proposition.
\begin{proposition}\label{prop.2}
For sufficiently large $T\geq1$, we have
\begin{equation}\label{eq.12}
\left|\lambda_{\chi}(n)-\lambda_{\chi}(n,T)\right|\leq\frac{3n^{2}}{2T^{2}}\left[\frac{1}{2\pi}T\log T+\left(\frac{1}{\pi}+\log\left(\frac{q}{2\pi e}\right)\right)T+\frac{1}{2}\right].
\end{equation}
\end{proposition}
\begin{proof}
Note that $\rho=\beta+i\gamma$, where $\beta$, $\gamma\in{\rb}$ and $0\leq\beta\leq1$. We have
$$\lambda_{\chi}(n)-\lambda_{\chi}(n,T)=\frac{1}{2}Re\left[ \sum_{|\gamma|> T}\left(2-\left(\frac{\rho-1}{\rho}\right)^{n}-\left(\frac{\rho}{\rho-1}\right)^{n}\right)\right].$$
Using a binomial expansion of the inner term in the sum in the right-hand side, we obtain
\begin{eqnarray}\label{eq.13}
Re\left[2-\left(\frac{\rho-1}{\rho}\right)^{n}-\left(\frac{\rho}{\rho-1}\right)^{n}\right]&=&Re\left[n\left(\frac{1}{\rho}+\frac{1}{1-\rho}\right)-\frac{n(n-1)}{2}\left(\frac{1}{\rho^{2}}+\frac{1}{(1-\rho)^{2}}\right)\right]\nonumber\\
&&+\ Re\left[\sum_{k=3}^{n}\left(_{k}^{n}\right)(-1)^{k-1}\left(\rho^{-k}+(1-\rho)^{-k}\right)\right].
\end{eqnarray}
We have
$$\frac{1}{1+\gamma^{2}}\leq Re\left(\frac{1}{\rho}+\frac{1}{1-\rho}\right)\leq\frac{1}{\gamma^{2}}$$
and
$$\frac{1}{1+\gamma^{2}}-\frac{2}{\gamma^{4}}\leq Re\left(\frac{1}{\rho^{2}}+\frac{1}{(1-\rho)^{2}}\right)\leq\frac{2}{\gamma^{2}}.$$
Suppose now that $|\gamma|\geq T\geq n$, then $\frac{(n-3)}{|\gamma|}...\frac{(n-k)}{|\gamma|}\leq1$ for all $n\geq k\geq3$. Then
$$\sum_{k=3}^{n}\left(_{k}^{n}\right)\frac{1}{|\gamma|^{k}}\leq2\frac{n^{3}}{|\gamma|^{3}}\sum_{k=3}^{\infty}\frac{1}{k!}=(2e-5)\frac{n^{3}}{|\gamma|^{3}}\leq\frac{n^{3}}{2|\gamma|^{3}}.$$
Therefore,
\begin{eqnarray}\label{eq.14}
Re\left[2-\left(\frac{\rho-1}{\rho}\right)^{n}-\left(\frac{\rho}{\rho-1}\right)^{n}\right]&\leq&\frac{n}{\gamma^{2}}+\frac{n^{2}-n}{\gamma^{2}}+\sum_{k=3}^{n}\left(_{k}^{n}\right)\frac{2}{|\gamma|^{k}}\nonumber\\
&\leq&\frac{n}{\gamma^{2}}+\frac{n^{3}}{2|\gamma|^{3}}\nonumber\\
&\leq&\frac{3n^{2}}{2|\gamma|^{2}}.
\end{eqnarray}
Then
$$\left|\lambda_{\chi}(n)-\lambda_{\chi}(n,T)\right|\leq\frac{3}{4}n^{2}\sum_{|\gamma|> T}\frac{1}{\gamma^{2}}.$$
We have
\begin{equation}\label{eq.15}
\frac{1}{2}\sum_{|\gamma|>T}\frac{1}{\gamma^{2}}\leq\int_{T}^{\infty}-\frac{d}{dt}\left[t^{-2}\right]_{t=x}\left(N_{\chi}(x)-N_{\chi}(T)\right)dx
=\int_{T}^{\infty}x^{-2}dN_{\chi}(x).
\end{equation}
Furthermore,
\begin{eqnarray}
\int_{T}^{\infty}x^{-2}dN_{\chi}(x)&=&\int_{T}^{\infty}x^{-2}\left[\frac{1}{2\pi}\log x+\frac{1}{2\pi}+\log\left(\frac{q}{2\pi e}\right)+\frac{1}{x}\right]dx\nonumber\\
&=&T^{-2}\left[\frac{1}{2\pi}\log T+\frac{1}{2\pi}T\log T+\frac{1}{2\pi}T+\log\left(\frac{q}{2\pi e}\right)T+\frac{1}{2}\right].\nonumber
\end{eqnarray}
Finally, we get
$$\left|\lambda_{\chi}(n)-\lambda_{\chi}(n,T)\right|\leq\frac{3}{2}\frac{n^{2}}{T^{2}}\left[\frac{1}{2\pi}T\log T+\left(\frac{1}{\pi}+\log\left(\frac{q}{2\pi e}\right)\right)T+\frac{1}{2}\right]$$
and  Proposition \ref{prop.2} follows.
\end{proof}
For our computations  at the end of this paper, we need to find $T_{0}$ such that $\left|\lambda_{\chi}(n)-\lambda_{\chi}(n,T)\right|\leq 10^{-k}$. To do so, it suffices to find $T_{0}$ such that
$$\frac{3n^{2}}{4\pi}\frac{\log T}{T}\leq \frac{10^{-k}}{3}\ \ \ \Leftrightarrow \ \ \ \frac{\log T}{T}\leq \frac{4\pi 10^{-k}}{9n^{2}}.$$
Using the theory of  the Lambert $W$ function, we get
$$T_{0}=-\frac{9n^{2}}{4\pi}W_{-1}\left(-\frac{4\pi}{9n^{2}}10^{-k}\right),$$
where $W_{-1}$ denotes the branch satisfying $W(x)\leq-1$ and $W(x)$ is the Lambert $W$ function which is defined to be the multivalued inverse of the function $w\longmapsto we^{w}$.

 \end{section}
\begin{section}{New formulas for the Li coefficients}\label{sec.4}
In this section, we give  new formulas for the Li coefficients under the Riemann hypothesis (integral and summation formula) which will be used to compute and verify the positivity of the Li coefficients $\lambda_{\chi}(n)$ under the Riemann hypothesis.\\

\noindent From (\ref{eq.1}) we have
\begin{equation}\label{eq.16}
\log\xi\left(\frac{z}{z-1},\chi\right)=\log\xi\left(\frac{1}{1-z},\chi\right)=\log \xi(0,\chi)+ \sum_{n=1}^{\infty}\lambda_{\chi}(n)\frac{z^{n}}{n}.
\end{equation}
The number $\lambda_{\chi}(n)$ does not depend on the choice of the logarithm. Rewrite  (\ref{eq.16}) at the point $z=-1$. Note that the region of convergence for this is an open disk of radius 2 centered at $z=-1$ and it encloses in particular the entirety of the closed unit disk, except for the point $z=1$ that is a pole of $\xi(\frac{1}{1-z},\chi)$. \\

Assume that the Riemann hypothesis
holds. Then, we have
\begin{eqnarray}\label{eq.17}
\log\xi\left(\frac{1}{1-z},\chi\right)&=&\log \xi(0,\chi)+ \sum_{n=1}^{\infty}\lambda_{\chi}(n)\frac{z^{n}}{n}\nonumber\\
&=&C_{\chi}(0)+\sum_{n=1}^{\infty}C_{\chi}(n)(z+1)^{n}.
\end{eqnarray}
Expanding $(z+1)^{n}$, we obtain
\begin{equation}\label{eq.18}
\lambda_{\chi}(n)=n\sum_{j=1}^{\infty}\left(_{n}^{j}\right)C_{\chi}(j).
\end{equation}
We have
\begin{equation}\label{eq.19}
N_{\chi}(T)=\frac{1}{\pi}Im\left(\log\xi_{\chi}\left(\frac{1}{2}+iT\right)\right)=\sum_{n=1}^{\infty}\frac{C_{\chi}(n)}{\pi}\left(Im\left(\frac{2\gamma+i}{2\gamma-i}+1\right)\right)^{n},
\end{equation}
where we have used the substitution $\frac{1}{2}+i\gamma=\frac{1}{1-z}$ or $z=\frac{2\gamma+i}{2\gamma-i}$. Then, since
$$
\left(Im\left(\frac{2\gamma+i}{2\gamma-i}+1\right)\right)^{n}=\frac{(4\gamma)^{n}}{(4\gamma^{2}+1)^{n/2}}\sin\left(n\tan^{-1}\frac{1}{2\gamma}\right)=2^{n}\cos^{n}\theta\sin(n\theta),$$
$$ \cos\theta:=\frac{2\gamma}{\sqrt{4\gamma^{2}+1}},
$$
we get
$$\pi  N_{\chi}(\gamma)=\sum_{n=1}^{\infty}C_{\chi}(n)2^{n}\cos^{n}\theta\sin(n\theta).$$
Using the identity
$$\int_{0}^{\pi/2}\cos^{n}\theta\sin(n\theta)\sin(2m\theta)d\theta=\frac{\pi}{2^{n+2}}\left(_{m}^{n}\right),\ \ m, n\in{\nb},$$
we deduce
$$\int_{0}^{\pi/2}\pi N_{\chi}(\gamma)\sin(2m\theta)d\theta=\sum_{n=1}^{\infty}C_{\chi}(n)\frac{\pi}{4}\left(_{m}^{n}\right).$$
Hence,
$$\sum_{n=1}^{\infty}C_{\chi}(n)\left(_{m}^{n}\right)=4\int_{0}^{\pi/2}N_{\chi}(\gamma)U_{m-1}(\cos(2\theta))\sin(2\theta)d\theta,$$
where $U_{m-1}$ are the Chebyschev polynomial of the second kind. Using that $$U_{m-1}(\cos\theta):=\frac{\sin(m\theta)}{\sin(\theta)},\ \ \ \cos(2\theta)=\cos^{2}\theta-\sin^{2}\theta=\frac{4\gamma^{2}-1}{4\gamma^{2}+1},$$ $$\sin(2\theta)d\theta=-2\cos\theta d(\cos\theta)$$ and that as $\gamma$ proceeds from 0 to $\infty$, $\theta$ subtends an angle from $\pi/2$ to $0$, we obtain
$$\sum_{n=1}^{\infty}C_{\chi}(n)\left(_{m}^{n}\right)=8\int_{0}^{\infty}N_{\chi}(\gamma)U_{m-1}\left(\frac{4\gamma^{2}-1}{4\gamma^{2}+1}\right)\times\frac{2\gamma}{\sqrt{4\gamma^{2}+1}}\times\frac{2}{(4\gamma^{2}+1)^{3/2}}d\gamma.$$
\noindent Therefore, from (\ref{eq.18}) we get for all $n\in{\nb}$ the following proposition.
\begin{proposition}\label{prop.3} Under the Riemann hypothesis, for $n\geq1$, we have
\begin{equation}\label{eq.20}
\lambda_{\chi}(n)=32\ n\ \int_{0}^{\infty}\frac{\gamma}{(4\gamma^{2}+1)^{2}}N_{\chi}(\gamma)U_{n-1}\left(\frac{4\gamma^{2}-1}{4\gamma^{2}+1}\right)d\gamma.
\end{equation}
\end{proposition}

Next, we give another formula for the Li coefficient. Recall that the function $N_{\chi}(T)$ is a real step function, increasing by unity each time a new critical zero is counted:
\begin{equation}\label{eq.21}N_{\chi}(T)=\sum_{\rho, Im(\rho)>0}\phi(T-Im(\rho))=\sum_{k=1}^{\infty}\alpha_{k}\phi(T-\gamma_{k}),
\end{equation}
where $\rho_{j}=\beta_{k}+i\gamma_{k},\ \gamma_{k}>0$ and $\phi(x-a)=1$ if $x\geq a$ and 0 if $x<a$. The zeros are ordered so that $\gamma_{k+1}>\gamma_{k}$ and the $\alpha_{k}$ counts the number of zeros with imaginary part $\gamma_{k}$ including the multiplicities.
Simplification of the integral formula (\ref{eq.20}) is stated in the following proposition.
\begin{proposition}\label{prop.4} Under the Riemann hypothesis, we have
$$\lambda_{\chi}(n)=2\sum_{k=1}^{\infty}\alpha_{k}\left(1-T_{n}\left(\frac{4\gamma_{k}^{2}-1}{4\gamma_{k}^{2}+1}\right)\right),\ \ n\in{\nb}.$$
\end{proposition}
\begin{proof}By (\ref{eq.21}), the formula (\ref{eq.20}) can be written as follows:
\begin{eqnarray}
  \lambda_{\chi}(n)&=&32n\sum_{k=1}^{\infty}\alpha_{k}\int_{0}^{\infty}\phi(\gamma-\gamma_{k})\frac{\gamma}{(4\gamma^{2}+1)^{2}}U_{n-1}\left(\frac{4\gamma^{2}-1}{4\gamma^{2}+1}\right)d\gamma\nonumber\\
&=&2n\sum_{k=1}^{\infty}\alpha_{k}\int_{\gamma_{k}}^{\infty}\frac{16\gamma}{(4\gamma^{2}+1)^{2}}U_{n-1}\left(\frac{4\gamma^{2}-1}{4\gamma^{2}+1}\right)d\gamma\nonumber\\
&=&2n\sum_{k=1}^{\infty}\alpha_{k}\left[\frac{1}{n}T_{n}(y)\right]_{\frac{4\gamma_{k}^{2}-1}{4\gamma_{k}^{2}+1}}^{1}\nonumber\\
&=&2\sum_{k=1}^{\infty}\alpha_{k}\left(1-T_{n}\left(\frac{4\gamma_{k}^{2}-1}{4\gamma_{k}^{2}+1}\right)\right),\nonumber
\end{eqnarray}
using the following relation between the Chebyshev polynomials of the second kind and the first kind
$$\int U_{n}(x)dx=\frac{1}{n+1}T_{n+1}(x).$$
\end{proof}
This is remarkable summation expression for the Li coefficients. We numerically evaluate
some of the first terms by the right hand side expression and find them to be indeed close to the required values of the Li coefficients. This is reassuring, and the results are presented in the tables below.\\

Under the Riemann hypothesis, from  the above arguments used in the proof of Propositions \ref{prop.3} and \ref{prop.4}, one can derive the following formula
$$\lambda_{\chi}(n,T)=2\sum_{k=1}^{N}\alpha_{k}\left(1-T_{n}\left(\frac{4\gamma_{k}^{2}-1}{4\gamma_{k}^{2}+1}\right)\right),$$
where $N=[N_{\chi}(T)]$ with $[x]=x-\{x\}$ and $\{x\}$ denotes the fractional part of $x$ (the last formula will be denoted $\lambda_{\chi}(n,N)$). Therefore, the latter formula allows one to estimate the error term $|\lambda_{\chi}(n)-\lambda_{\chi}(n,N)|$ in Proposition \ref{prop.4} by evaluating directly the partial Li coefficients as in Proposition \ref{prop.2}.

\end{section}
\begin{section}{Numerical computations}\label{sec.5}

In this section,  we compute and verify the positivity of the  values of $\lambda_{\chi}(n)$ unconditionally or under the Riemann hypothesis. We first compute unconditionally (without assuming the Riemann hypothesis) $\tilde{\lambda}_{\chi}(n,M)$ by using equation (\ref{eq.3}) and computing prime numbers up to $M$ (see. Section \ref{sec.2}). We also compute under the Riemann hypothesis \begin{equation}\label{eq.22}\lambda_{\chi}(n,N)=2\sum_{k=1}^{N}\alpha_{k}\left(1-T_{n}\left(\frac{4\gamma_{k}^{2}-1}{4\gamma_{k}^{2}+1}\right)\right),\ \hbox{with}\  N=10^{4}.
 \end{equation}
 Furthermore, we carried out the calculations for several examples of characters. Some illustrative examples are cited below. We restricted the  tables below for $n\leq 40$. However, one  can find the other values of $n>40$ represented in the  graphs 1-4. \\

 \noindent {\bf Remark.} In fact, by the summation formula (\ref{eq.22}), we could compute more coefficients $\lambda_{\chi}(n)$ with less time consuming way than by the arithmetic formula (\ref{eq.3}), where computation of the first 50 coefficients lasted more than a week.\\

 \noindent Based on the  tables below, we conjecture the following result.\\
{\bf Conjecture. The coefficients $\lambda_{\chi}(n)$ are positive and increasing in $n$.}\\

\noindent This conjecture was partially numerical verified for the case of the Riemann zeta function (see \cite[Appendix D]{2} and \cite{8}) and by the authors in a work in progress for the Hecke $L$-functions \cite{13}.
 {\tiny \begin{center}
\begin{tabular}{ | l || c | c |c ||c||c|c|c|}
  \hline
$\chi$(mod3)\\
    $n$ & \ $\tilde{\lambda}_{\chi}(n,M)\  $ &\ $\lambda_{\chi}(n,N)$\ &\  \ $n$\ \  &\  $\tilde{\lambda}_{\chi}(n,M)$\  &\ $\lambda_{\chi}(n,N)$\ \\ \hline
    1  & 0.05316& 0.056442 &19 & 17.18050  & 17.16170     \\ \hline
  2  & 0.22763 & 0.22542   &20 & 18.58480  & 18.69100   \\ \hline
  3  & 0.14844  & 0.50592&21 & 20.01400  & 20.24310       \\ \hline
  4  & 0.89344   & 0.89624&22 & 21.46700   & 21.81300        \\ \hline
  5  & 1.35725  & 1.39404&23 & 22.94280  & 23.39600       \\ \hline
  6  & 2.12951  & 1.99635&24 & 24.44030  & 24.98820         \\ \hline
  7  & 2.98573  & 2.69962&25 & 25.95870  & 26.58590       \\ \hline
  8  & 3.91334  & 3.49978& 26 & 27.49700  & 28.18600        \\ \hline
  9  & 4.40970    & 4.39225&27 & 29.05460  & 29.78580        \\ \hline
  10 & 5.94841  & 5.37202&28 & 30.63070  & 31.38330     \\ \hline
  11 & 7.04344   & 6.43371&29 & 32.22460  &  32.97700        \\ \hline
  12 & 8.18382   & 7.57163&30 & 33.83580  & 34.56580        \\ \hline
  13 & 9.36580   & 8.77987&31 & 35.46370   & 36.14940        \\ \hline
  14 & 10.58620  & 10.05230 &32 & 37.10770  & 37.72780       \\ \hline
  15 & 11.84230  & 11.38280&33 & 38.76730  & 39.3014  \\ \hline
  16 & 12.81150  & 12.76510&34 & 40.44210  & 40.87120           \\ \hline
  17 & 14.45250  & 14.19300&35& 42.13150  & 42.43870     \\ \hline
  18 & 15.80260  & 15.66050&36&43.83530 &44.00550 \\ \hline

      \end{tabular}

\end{center}}

\begin{center}
\begin{figure}
\centering
\includegraphics[height=8cm]{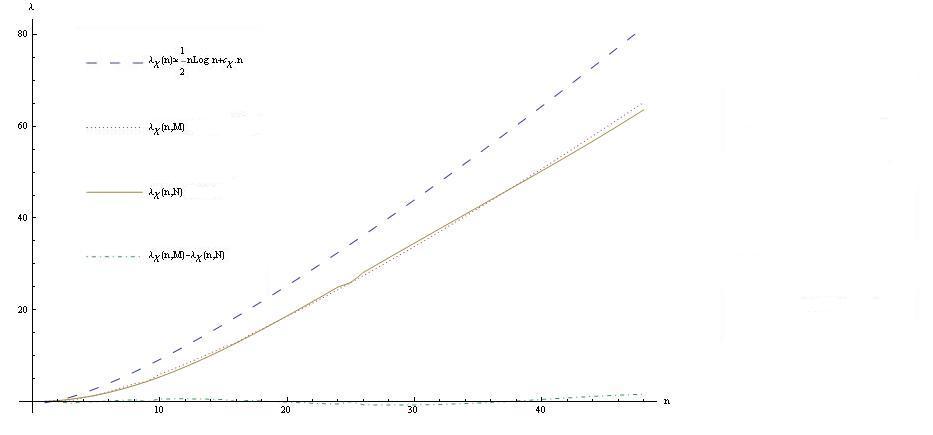}
\caption{ Case of $\chi$ (mod3)}
\label{fig:1}
\end{figure}
 {\tiny \begin{center}
\begin{tabular}{ | l || c | c |c ||c||c|c|c|}
\hline
  $\chi$(mod5)\\
   $n$ & \ $\tilde{\lambda}_{\chi}(n,M)\  $&\ $\lambda_{\chi}(n,N)$\ &\  \ $n$\ \  &\  $\tilde{\lambda}_{\chi}(n,M)$\   &\ $\lambda_{\chi}(n,N)$\ \\ \hline
  1  & 0.13183 & 0.08562  &21 & 25.37770   & 26.21450    \\ \hline
  2  & 0.29872 & 0.34152 &22 & 27.08610   & 27.92160    \\ \hline
  3  & 0.91468 & 0.76482  &23 & 28.81730   & 29.60960      \\ \hline
  4  & 1.58476   & 1.35081    &24 & 30.57020  & 31.27780     \\ \hline
  5  & 2.63432   & 2.09300  &25 & 32.34400    & 32.92720   \\ \hline
  6  & 3.66199  & 2.98332   &26 & 34.13770  & 34.56020   \\ \hline
  7  & 4.77362 & 4.01225   &27 & 35.95070  & 36.18030   \\ \hline
  8  & 5.95664   & 5.16902   &28 & 37.78220   & 37.79200       \\ \hline
  9  & 7.06010   & 6.44188    &29 &39.63160    &39.40090     \\ \hline
  10 & 8.50254  & 7.81828  &30 &41.49820  & 41.01320    \\ \hline
  11 & 9.85298   & 9.28519   &31&43.38150   &42.63540 \\ \hline
  12 & 11.24880   & 10.82930   &32 & 45.28090  & 44.2746    \\ \hline
  13 & 12.68620   & 12.43740    &33 & 47.19590   & 45.93760    \\ \hline
  14 & 14.16200   & 14.09650   &34& 49.12610&47.63100 \\ \hline
  15 & 15.67350  & 15.79410 &35&51.07100&49.36130\\ \hline
  16 & 17.68370   & 17.51860         &36& 53.03020&51.13410\\ \hline
  17 & 14.45250   & 19.25930 &37&55.00320&52.95430\\ \hline
  18 & 20.40000      & 21.00670&38&56.98980&54.82600\\ \hline
  19 & 22.03340   & 22.75260   &39&58.98950&56.75210\\ \hline
  20 & 23.69300   & 24.49030  &40&61.00210&58.73450\\ \hline

\end{tabular}
\end{center}}
\begin{figure}
\centering
\includegraphics[height=7cm]{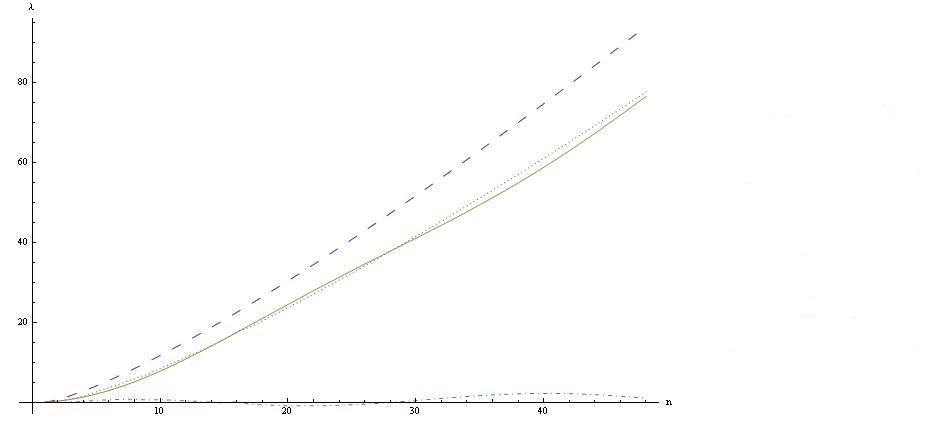}
\caption{Case of $\chi$ (mod5)}
\label{fig:1}
\end{figure}
 {\tiny \begin{center}
\begin{tabular}{ | l || c | c |c ||c||c|c|c|}
  \hline
$\chi$(mod20)\\
    $n$ & \ $\tilde{\lambda}_{\chi}(n,M)$\   &\ $\lambda_{\chi}(n,N)$\ &\  \ $n$\ \  &\  $\tilde{\lambda}_{\chi}(n,M)$\  &\ $\lambda_{\chi}(n,N)$\ \\ \hline
  1  & 0.695021 & 0.319128  &21 & 39.93370  & 41.70260 \\ \hline
  2  & 1.68502   & 1.24419   &22 & 42.33530  & 44.31350 \\ \hline
  3  & 2.99412   & 2.68343   &23 & 44.75970   & 46.59570\\ \hline
  4  & 4.48123   & 4.50032     &24 & 47.20580   & 48.55720 \\ \hline
  5  & 6.10005   & 6.53527  &25 & 49.67270  & 50.26430 \\ \hline
  6  & 7.82087   & 8.63067  &26 & 52.15960   & 51.83150 \\ \hline
  7  & 9.62565   & 10.65500   &27 & 54.66570  & 53.403100 \\ \hline
  8  & 11.50180   & 12.52230  &28 &57.19040   &55.12930  \\ \hline
  9  & 13.32220  & 14.20280    &29&59.73290&57.14130 \\ \hline
  10 & 15.43400  & 15.72450 &30&62.29260&59.52940 \\ \hline
  11 & 17.47760  & 17.16450 &31&64.86910&62.32740 \\ \hline
  12 & 19.56650  & 18.63130 &32&67.46160&65.50710  \\ \hline
  13 & 21.69710  & 20.24300 &33&70.06980&68.98220 \\ \hline
  14 & 23.86610   & 22.10320 &34&72.69310   &72.62260  \\ \hline
  15 & 26.07070  & 24.28030&35&75.33110&76.27560\\ \hline
  16 & 28.54690  & 26.79300&36&77.98350&79.79060\\ \hline
  17 & 30.57800    & 29.60520&37&80.64970&83.04340\\ \hline
  18 & 32.87670   & 32.63050&38&83.32940&85.95580\\ \hline
  19 & 35.20320   & 35.74610&39&86.02230 &88.50750\\ \hline
  20 & 37.55600   & 38.81360&40&88.72800&90.73760 \\ \hline

      \end{tabular}
\end{center}}
\begin{figure}
\centering
\includegraphics[height=7cm]{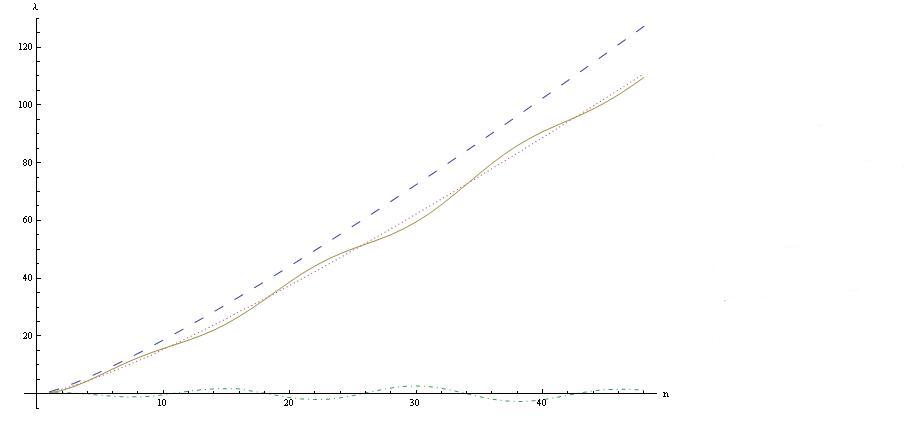}
\caption{Case of $\chi$ (mod20)}
\label{fig:1}
\end{figure}
 {\tiny \begin{center}
\begin{tabular}{ | l || c | c |c ||c||c|c|c|}
\hline
  $\chi$(mod60)\\
   $n$ & \ $\tilde{\lambda}_{\chi}(n,M)\  $ &\ $\lambda_{\chi}(n,N)$\ &\  \ $n$\ \  &\  $\tilde{\lambda}_{\chi}(n,M)$\  &\ $\lambda_{\chi}(n,N)$\ \\ \hline
  1  & 1.12226 &0.48626 &21 & 51.46920 &50.88960  \\ \hline
  2  & 2.78363 &1.86950    &22 & 54.42010 &52.52830  \\ \hline
  3  & 4.64204 &3.94169    &23 & 57.39370 &54.44350  \\ \hline
  4  & 6.83662 &6.41363     &24 & 60.38910  &56.86290  \\ \hline
  5  & 8.84658 &8.98530     &25 & 63.40530 &59.89590 \\ \hline
  6  & 11.11670 &11.41720     &26 & 66.44150 &63.50750  \\ \hline
  7  & 13.47080 &13.58380    &27 & 69.49700 &67.53000 \\ \hline
  8  & 15.89630&15.49640  &28 & 72.57090 &71.70750  \\ \hline
  9  & 18.06830&17.28820   &29 & 75.6628&75.7637  \\ \hline
  10 & 20.92710 &19.16770  &30 & 78.77180&79.47310  \\ \hline
  11 & 23.52000   &21.35250 &31 & 81.89750 &82.71770  \\ \hline
  12 & 26.15820 &24.00100   &32 & 85.03940 &85.51520   \\ \hline
  13 & 28.83810 &27.16160    &33 & 88.19690 &88.00960  \\ \hline
  14 & 31.55630 &30.75170    &34 & 91.36950&90.42920  \\ \hline
  15 & 34.31030 &34.57380&35 & 94.55690 &93.02160 \\ \hline
    16 & 37.56690  &38.36300& 36 & 97.75850&95.98430 \\ \hline
  17 & 39.91620 &41.85530&37 & 100.97400 &99.40850 \\ \hline
  18 & 42.76420 &44.85610&38 & 104.20300 &103.25300 \\ \hline
  19 & 45.64000   &47.29300&39 & 107.44500 &107.35200 \\ \hline
  20 & 48.54210 &49.23760&40 & 110.70000   &111.46000 \\ \hline
     \end{tabular}
\end{center}}
\begin{figure}
\centering
\includegraphics[height=7cm]{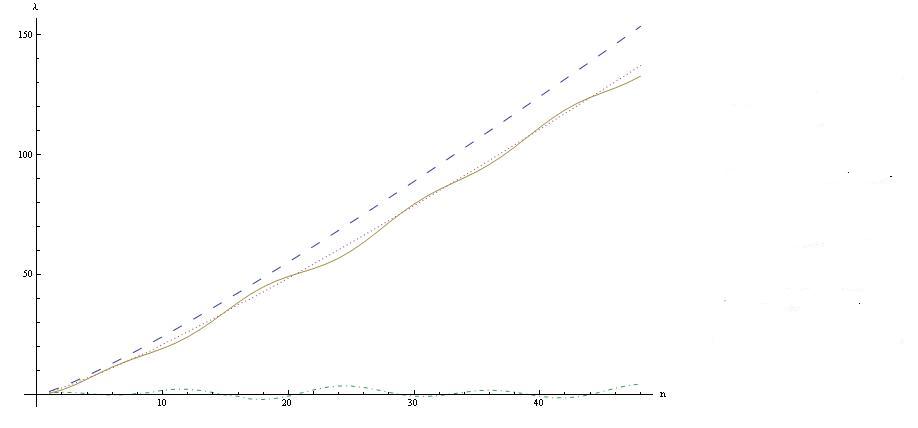}
\caption{Case of $\chi$ (mod60)}
\label{fig:1}
\end{figure}
\end{center}
{\bf Acknowledgements}.  The authors would like to thank   Maciej Radziejewski and Mark Coffey  for their  many valuable comments about the
published paper.

\end{section}

%\vspace{1cm}
%\affiliationoneone{% in this example, two authors share an institution
%  Sami Omar\\
%    Faculty of Science of Tunis\\Department of Mathematics\\2092 Tunis, Tunisia
%\email{sami.omar@fst.rnu.tn}\\

 % in this example, two authors share an institution
% \noindent  Raouf Ouni\\
%  Faculty of Science of Tunis\\Department of Mathematics\\2092 Tunis, Tunisia
 %  \email{raouf.ouni@fst.rnu.tn}}\\

%\noindent  \affiliationonetwo{
%\noindent  Kamel Mazhouda\\
 %  Faculty of Science of Monastir\\Department of Mathematics\\5000 Monastir\\Tunisia
  % \email{kamel.mazhouda@fsm.rnu.tn}}
% Important: Do not put any empty line here.
% Use \affiliationthree{} for any address positioned under \affiliationone
% Use \affiliationfour{}  for any address positioned under \affiliationtwo

\end{document}